\documentclass[12pt]{amsart}

\usepackage{
amsfonts,
latexsym,
amssymb,
}

\newcommand{\labbel}{\label}

\newtheorem{theorem}{Theorem}[section]

\newtheorem{proposition}[theorem]{Proposition} 
\newtheorem{cor}[theorem]{Corollary} 
\newtheorem{corollary}[theorem]{Corollary} 
\newtheorem{fact}[theorem]{Fact}

\theoremstyle{definition}
\newtheorem{definition}[theorem]{Definition}
\newtheorem{definitions}[theorem]{Definitions}

\theoremstyle{remark}
\newtheorem{remark}[theorem]{Remark}
 
\newtheorem{example}[theorem]{Example}

\newcommand{\brfrt}{\hspace{0 pt}}
\DeclareMathOperator{\CAP}{CAP}
\begin{document}
 
\title{A very general covering property}

\author{Paolo Lipparini} 
\address{Dipartimento di Matematica\\Viale del Ricoprimento Scientifico\\II Universit\`a di Roma (Tor Vergata)\\I-00133 ROME ITALY}
\urladdr{http://www.mat.uniroma2.it/\textasciitilde lipparin}

\thanks{We thank anonymous referees for many helpful
suggestions which greatly helped improving the paper. We thank our students from Tor Vergata University for stimulating questions. We thank Enna Andrici for encouragement.}



\begin{abstract}
We introduce a general notion of covering property,
of which many classical definitions are particular instances.
Notions of closure under various sorts of convergence, or, more generally,
under taking kinds of accumulation points, are shown to be equivalent to
a covering property in the sense considered here (Theorem \ref{r}). Conversely,
every covering property is equivalent to the existence 
of  appropriate kinds of
accumulation points for arbitrary sequences on some fixed index set (Corollary \ref{covisacc}).

We  discuss corresponding notions related to sequential compactness, and to pseudocompactness,
or, more generally, properties connected with the existence of limit points
of sequences of  subsets. In spite of the great generality of our
treatment, many results here appear to be new even in very special cases,
such as $D$-\brfrt compactness and $D$-\brfrt pseudocompactness, for $D$ an ultrafilter, and 
weak (quasi) $M$-(pseudo)-compactness, for $M$ a set of ultrafilters, as well as 
for $[ \beta , \alpha   ]$-\brfrt compactness, with $\beta$ and $\alpha$ ordinals.
\end{abstract} 
 
\maketitle

\section{Introduction} \labbel{intro}
 
``Covering property'' in the title refers to a property of the form
``every open cover has a subcover by a 
tractable
 class of elements''.
The most general and easiest form of establishing what ``tractable'' is to be intended  simply
amounts to enumerate those sets which are to be considered \emph{tractable}.
We are thus led to the following definition, where
$\mathcal P(A)$ denotes the set of all subsets of the set $A$.

\begin{definition} \labbel{BAcpn} 
\cite[Definition 7.7]{ordcpn} If $A$ is a set, and $B \subseteq \mathcal P(A)$,
  we say that a topological space
$X$ is \emph{$[ B, A ]$-\brfrt compact} if and only if,
whenever $(O_ a) _{ a \in A} $ is a sequence of 
open sets of $X$ such that 
$ \bigcup  _{ a \in A } O_ a = X $,
then there is $H \in B $  such that
$ \bigcup  _{ a \in H } O_ a= X $.
\end{definition}   

Of course, (full) \emph{compactness} is the particular case when 
 $A$ is infinite and arbitrarily large, and $B$ is the set of all finite subsets of $A$.
If in the above sentence we replace finite by countable, we get \emph{Lindel\"ofness}. On the other hand, if we instead restrict only
to countable $A$, we get \emph{countable compactness}. More general examples
of cardinal (and ordinal) notions reducible to $[ B, A ]$-\brfrt compactness shall be presented below.

We first show how to produce 
 counterexamples to 
$[B , A    ]$-\brfrt compactness
in a standard way.

\begin{example} \labbel{b} 
Suppose that $A$ is a set, $ B \subseteq \mathcal P(A)$, 
$B$ is nonempty,
and $A \not \in B$
(the assumption $A \not \in B$ is necessary by Fact \ref{fac}(1) below).

(a) As a typical counterexample to  
$[ B , A  ]$-\brfrt compactness,
we can exhibit  $B$ itself, with the  topology
a subbase of which
consists of the sets
$a^{\nless}=\{ H \in B \mid a \not\in H \}$,
$a$ varying in $A$.

With the above topology, 
$B$ is not 
$[ B , A  ]$-\brfrt compact,
as the $a^{\nless}$s themselves
witness.
Indeed, the  $a^{\nless}$s are a cover of $B$, since 
$A \not \in B$.
However, for every  $H \in B$, 
$ (a^{\nless}) _{a \in H} $ is not a cover
of $B$, since $H$ belongs to 
no $a^{\nless}$, for $a \in H$.   

We believe that, in a sense still to be made precise,
$B$ with the above topology
 is the typical example of a non  $[ B , A  ]$-\brfrt compact topological space.
This is suggested by particular cases concerning ordinal compactness,
see \cite[Theorems 5.4 and 5.7]{ordcpn}.

 Notice that $B$, with the above topology, is $T_0$, but, in general, not even $T_1$. 
However, the example can be turned into a Tychonoff  topological space 
by introducing a finer topology
 as in (b) below.

Observe that $\mathcal P(A)$ is in a bijective correspondence,
via characteristic functions, with 
$ { {^A} } \{ 0, 1\}  $,
the set of all functions from
$A$ to 
$ \{ 0, 1\}  $,
hence with the product of 
$A$-many copies of 
$  \{ 0, 1\}  $. 
Via the above identifications,
if we give to 
$  \{ 0, 1\}  $
the topology in which 
$  \{ 0\}  $
is open, 
but $  \{  1\}  $
is not open, then the topology described above
is the subspace topology induced
on $B$ by the (Tychonoff)  
product topology on 
$ { {^A} } \{ 0, 1\}  $.

(b) 
If we instead give to 
$  \{ 0, 1\}  $
the discrete topology, 
then the subset topology induced on $B$ 
by the topology on 
$ { {^A} } \{ 0, 1\}  $ makes $B$ a Tychonoff topological space,
which is still a counterexample to $[ B , A  ]$-\brfrt compactness.
This latter topology, too,  admits an explicit description: it is the topology
a subbase of which consists of the sets 
which have either the form
$a^{\nless}=\{ H \in B \mid a \not\in H \}$,
or
$a ^{<} =\{ H \in B \mid a \in H \}$,
for some
$a \in A$.

If $B$ is closed under symmetrical difference, then, with this topology,
 $B$ inherits from $ { {^A} } \{ 0, 1\}  $
 the structure of a topological group.
If $B$ is closed both under finite unions and finite intersections, then
  $B$ inherits from $ { {^A} } \{ 0, 1\}  $
 the structure of a topological  lattice. 
\end{example}  

We now consider some more specific instances of Definition \ref{BAcpn}. 

The most general form of a covering notion involving cardinality
as a measure of ``tractability'' is 
\emph{$[ \mu, \lambda  ]$-\brfrt compactness},
where $\mu$ and $\lambda$ are cardinals.
It is the particular case of  Definition \ref{BAcpn} 
 when $A= \lambda $ and 
$B = \mathcal P _{\mu}( \lambda ) $
is the set of all subsets of $ \lambda $ of cardinality $<\mu$. 
The notion of $[ \mu, \lambda  ]$-\brfrt compactness originated in the 
20's in the past century \cite{AU}, and thus has a very long history. See, e.~g., 
\cite{Ga,Go,Sm,St,Vlnm,Vfund,Vhstt,Vecgt} for results and references.

In \cite{ordcpn} we generalized cardinal compactness to ordinals, that is, we 
considered the particular case of Definition \ref{BAcpn} 
in which $A$ is an ordinal (or, anyway, a well-ordered set),
and the ``tractability'' of some subset $H$ of $A$ 
is measured by considering  the  order type of $ H$.
In more detail, for $\alpha$ and $\beta$ ordinals,
\emph{$[ \beta , \alpha   ]$-\brfrt compactness}
is obtained from Definition \ref{BAcpn}  by letting
$A= \alpha $, and letting $B$ equal to the set of all subsets of $ \alpha $ 
having order type $<\beta$. 
The notion is interesting, since one can prove many 
non trivial results of the form  ``every $[ \beta , \alpha   ]$-\brfrt compact
space is $[ \beta' , \alpha'   ]$-\brfrt compact'', for various ordinals, while only trivial 
results of this kind hold, when restricted to cardinals.  
Moreover, there are examples of spaces satisfying exactly the same 
$[ \mu, \lambda  ]$-\brfrt compactness (cardinal) properties, but which 
behave in a very different way as far as (ordinal) $[ \beta , \alpha  ]$-\brfrt compactness is concerned. Just to present the simplest possible example,
if $\kappa$ is a regular uncountable cardinal, then $\kappa$, with the order topology,
is $[ \kappa + \kappa , \kappa + \kappa  ]$-\brfrt compact, but the disjoint 
union of two copies of $\kappa$ is not $[ \kappa + \kappa , \kappa + \kappa  ]$-\brfrt compact (here $+$ denotes ordinal sum). 
Furthermore, there are 
many rather deep connections among $[ \beta , \alpha  ]$-\brfrt compactness, cardinalities and separation properties of spaces.
In  \cite{ordcpn} we also introduced an ordinal version of the
Lindel\"of number of a topological space, and showed that this ordinal version
gives much more informations about the space than the cardinal version.

So far, we have not yet provided really strong motivations in favor
of the great generality of Definition \ref{BAcpn}. Indeed, at first sight,
the ordinal version of compactness,  that is, $[ \beta , \alpha   ]$-\brfrt compactness, 
appears to be  a quite very sensitive and fine notion, 
well suited for exactly measuring the covering properties 
enjoyed by some topological space. However,  other interesting 
properties naturally insert themselves into the general framework given by
Definition \ref{BAcpn}. In fact, besides considering 
$[ \beta , \alpha   ]$-\brfrt compactness, we reached the notion 
of $[ B, A   ]$-\brfrt compactness after a careful look at the proposition below, which characterizes $D$-compactness.

Recall that if $D$ is an ultrafilter, say over some set $I$, then a topological space $X$ is said to be
$D$-\emph{compact} if and only if every sequence 
$(x_i)_{i \in I}$ 
of elements of $X$
$D$-converges to some point of $x $, where
a sequence  $(x_i)_{i \in I}$ is said to 
$D$-\emph{converge}  to some point $x \in X$ if and only if
$ \{ i \in I \mid  x_i \in U\} \in D$, for every neighborhood
$U$ of $x$ in $X$.

In \cite [Corollary 34]{tapp2} 
we proved the following proposition, which is also
a particular case of Theorem \ref{r} below (see Remark \ref{d}).

\begin{proposition} \labbel{DAB}
Let $D$ be an ultrafilter over $I$.
A topological space 
$X$ is $D$-\brfrt compact if and only if, 
for every open cover $(O_Z) _{Z \in D} $ of $X$,
there is some $i \in I$ such that  
$(O_Z) _{i \in Z } $
is a cover of $X$.
 \end{proposition} 

Thus also $D$-compactness is equivalent
to a covering property,
namely, the particular case of Definition \ref{BAcpn}
in which $A$ is $D$ itself, 
and $B= \{ i^< \mid i \in I\} $,
where, for $i \in I$, we put  $i^<=\{ Z \in D \mid i \in Z\}$.
In words, $B$ is the set of all the intersections of $D$
with some principal ultrafilter.
Hence, in the sense of $D$-compactness,
 being ``tractable'' means (having indices) lying
in the intersection of $D$ with some principal ultrafilter.

Reflecting on the above example, we soon realized that
many other conditions asking closure under appropriate
types of convergence are equivalent to covering properties.
Furthermore, this is the case also for 
 the existence of kinds of accumulation points,
 as we shall show in Section \ref{convcov}. Historically, the interplay
between covering properties and accumulation properties
has been a central theme in topology, starting from \cite{AU},
if not earlier. In this respect, see also the discussion in Remark \ref{classical}.

Also a generalization of $D$-compactness, weak $M$-compactness,
involving a set $M$ of ultrafilters, is equivalent to a covering property,
as will be shown in Corollary \ref{wc}. 
See Corollary \ref{quasicp} for a characterization of a further related notion: quasi $M$-compactness. 

If in Definition \ref{BAcpn} we take $B= \mathcal P(A) \setminus \{ A \} $, then
a counterexample to 
$[B , A    ]$-\brfrt compactness is what is usually called
an \emph{irreducible} (or \emph{minimal}) cover. 
Irreducible covers, as well as spaces
 in which every cover can be refined to a
(possibly finite)  irreducible cover
have been the object of some study.
See \cite{AD,L} and further references there.
In a sense, an infinite irreducible cover produces a maximal form
of incompactness. Indeed, e.~g.~ by Fact \ref{fac}(2)(6) below in contrapositive form,
if some topological space $X$ has an irreducible cover of cardinality $\lambda$,
then     $X$ is not $[B , A    ]$-\brfrt compact, for every set $A$ 
such that $| A| \leq \lambda $, 
and every $B \subseteq \mathcal P (A)$ such that 
  $A \not \in B$.

If $X$ is a $T_1$  topological space which is not
countably compact, then any open cover witnessing countable incompactness
can be refined to an irreducible  countably infinite open cover. This follows,
for example, from the proof  of \cite[Lemma 6.4]{ordcpn}. Compare also with 
\cite[Theorem 2.1]{AD}.  Thus we get the following
Proposition.

\begin{proposition} \labbel{t1}
 For a 
$T_1$   topological space $X$, the following conditions are equivalent 
  \begin{enumerate}
    \item 
$X$  is not
countably compact.
\item
  $X$ is not
 $[ \mathcal P(A) \setminus \{ A \} , A    ]$-\brfrt compact,
for every  countable nonempty set $A$.
\item
$X$ is not
 $[B, A    ]$-\brfrt compact,
for some countably infinite set $A$,
and some $B \subseteq \mathcal P(A)$ such that 
 $B $ contains all finite subsets of $A$.  
  \end{enumerate} 
 \end{proposition} 

The above equivalences do not generalize to uncountable cardinals.
The space $\kappa$, with the order topology, is not
$[ \kappa, \kappa ]$-compact, but it is   
$[ \kappa+ \omega , \kappa+ \omega  ]$-compact \cite[Example 3.2(3)]{tapp2}
(here $+$ denotes ordinal sum).
Moreover the hypothesis that $X$ is $T_1$ is necessary,
by \cite[Example 3.2(2)]{tapp2}.

Though simple, Definition \ref{BAcpn} unifies 
many disparate situations, and allows for the possibility of proving
some interesting and non trivial results, which sometimes are new and 
useful even in very particular cases.

\smallskip

When restricted to (cardinal) $[ \mu , \lambda    ]$-\brfrt compactness,
some of the results presented in this note might be seen as a revisitation of known results.
They are new in the case of
(ordinal)
$[ \beta , \alpha   ]$-\brfrt compactness.
Actually, the study of properties 
of $[ \beta , \alpha   ]$-\brfrt compactness
has been the leading motivation for the present research. 
Restricted to this particular case, this note may be seen 
as a continuation of \cite{ordcpn}.
As soon as we realized that the
results naturally fit into a more general setting, 
with no essential further technical complication,
we decided to present them in  their more general form.

As far as $D$-compactness, and other notions of convergence are  concerned, the results presented here
can improve shedding new light into the subject.
In particular, they hopefully provide a new point of view about the relationship between
convergence, accumulation  and covering properties.

It might be of some interest the fact that there is also  a version  for notions related to 
pseudocompactness. As well known, for Tychonoff spaces, there is an equivalent formulation of pseudocompactness which
involves open covers: a Tychonoff space is \emph{pseudocompact} if and only if every countable
open cover has a subset with dense union.  Here the premise is the same as in
countable compactness, with a weakened conclusion.
Definition \ref{BAcpn}, too, can be modified in the same way (Definition \ref{ABO}), and
essentially all the results we prove for $[B , A    ]$-\brfrt compactness
have a version for this pseudocompact-like notion. The notion
of convergence (or accumulation) of a sequence of \emph{points}
 shall be replaced
with  notions of limit points of a sequence of \emph{subsets}. 

Furthermore, in Section \ref{sequential}, we present  variations which include 
covering properties equivalent to sequential compactness, sequential pseudocompactness, 
quasi  $M$-compactness,  the Menger property and the Rothberger property.
Many other notions can be obtained as particular cases of Definition \ref{ababg}. 
 Definition \ref{ababg} probably deserves further study.

\smallskip

We assume no separation axiom, unless otherwise specified.

\section{Equivalents of a covering property} \labbel{top} 

In this section we show that, for every $B$ and $A$ as in Definition \ref{BAcpn},
there are many equivalent formulations of $[ B, A ]$-\brfrt compactness.
In particular, it can be characterized by a sort of accumulation property,
in a sense which shall be explicitly described in the next section.
Parts of the results presented in this section are known for $[ \mu , \lambda  ]$-\brfrt compactness, hence, in particular, for countable compactness, Lindel\"ofness etc. They are new 
for (ordinal) $[ \beta , \alpha  ]$-\brfrt compactness, 
and for other general notions of compactness.

We begin with a trivial but useful fact.

\begin{fact} \labbel{factclosed}
A topological space  is $[ B, A ]$-\brfrt compact if and only if, for every sequence $(C_ a) _{ a \in A} $  of 
closed sets, if
$ \bigcap  _{ a \in H } C_ a \not= \emptyset  $,
for every  $H \in B $, then
$ \bigcap  _{ a \in A } C_ a \not= \emptyset  $.
 \end{fact}

\begin{proof}
Immediate from the definition of $[ B, A ]$-\brfrt compactness, 
in contrapositive form, and by taking complements.
\end{proof}

We now state some other easy facts about 
$[ B, A ]$-\brfrt compactness.

\begin{fact} \labbel{fac}
Suppose that $X$ is a topological space, $A$ is a set, and $B ,  B'\subseteq \mathcal P(A)$.
\begin{enumerate}    
\item  
If $A \in B $, then every topological space  is
$[ B, A ]$-\brfrt compact. In particular,
 every topological space   is $[ \{ A \}, A ]$-\brfrt compact.
\item
If $B \subseteq B'$,
 and  $X$ is $[ B, A ]$-\brfrt compact, 
then $X$ is $[ B', A ]$-\brfrt compact.
\item
More generally, if, for every $H \in B$, there is 
$H' \in B'$ such that $H \subseteq H'$, 
then every  
 $[ B, A ]$-\brfrt compact topological space 
 is $[ B', A ]$-\brfrt compact.
\item
If $X$ is $[ B , A  ]$-\brfrt compact, and $A' \subseteq A$, then
$X$ is $[ B| _{A'}  , A'  ]$-\brfrt compact,
where $B| _{A'}= \{ H \cap A' \mid  H \in B\}  $.
\item
Suppose that, for every $H \in B$,
$D_H \subseteq \mathcal P(H)$,
and let $D= \bigcup _{H \in B}D_H $.
If $X$ is $[ B , A  ]$-\brfrt compact, and 
$[ D_H , H  ]$-\brfrt compact, for every  $H \in B$,
then $X$ is
$[ D , A  ]$-\brfrt compact.
\item
If $C$ is a set, $f:C \to A$
is a surjective function,
and $D= \{ f ^{-1} (H) \mid H \in B\}$,
then every $[ B , A  ]$-\brfrt compact
topological space is $[ D , C ]$-\brfrt compact.
 \end{enumerate}
 \end{fact}   

Fact (5) above is a broad generalization of standard  arguments, e.~g., the argument showing
that Lindel\"ofness and countable compactness imply compactness.

Fact (6) follows immediately from the fact that a union of open sets 
is still open.
Indeed, if 
$(O_c) _{c \in C} $ is an open cover of $X$, 
then  
$(Q_a) _{a \in A} $ is an open cover of $X$,
where  
$Q_a =\bigcup_{f(c)=a}O_c$,
for $a \in A$ (using the assumption that $f$ is surjective). 

\begin{remark} \labbel{noloss} 
Let us say that  $B \subseteq  \mathcal P(A)$ is \emph{closed under
subsets} if and only if, whenever   $H' \in B$ and $H \subseteq H'$, 
then $H \in B$. 
Notice that, by (2) and (3) above, if 
$B' \subseteq \mathcal P(A)$
and $ B $ is the smallest subset of $ \mathcal P(A)$
which contains $B'$ and is closed under 
subsets, then a topological space $X$ is 
 $[ B, A ]$-\brfrt compact if and only if it is
 $[ B', A ]$-\brfrt compact.
Thus, in the definition of  $[ B, A ]$-\brfrt compactness,
it is no loss of generality to consider only  
those $B$ which are closed under subsets.
\end{remark}   

If $X$ is a topological space, and $ P \subseteq X$, we denote by $ \overline{P} $ the closure 
of $P$ in $X$, and by $P^\circ$ its interior. The topological space  in which we are taking closure and interior
shall always be clear from the context.
  
If $B \subseteq \mathcal P(A)$ and $a \in A$,
we let $a^<_B= \{ H \in B \mid a \in H\} $.

\begin{theorem} \labbel{e} 
Suppose that $A$ is a set, $B \subseteq \mathcal P(A)$, and $X$ is a topological space. Then the following conditions are equivalent.
\begin{enumerate} 
\item
$X$ is $ [ B, A]$-\brfrt compact.
\item
For every  sequence  $( P _ a ) _{ a \in A} $
of subsets of $X$, if,
for every $H \in B$, 
$ \bigcap _{ a \in H}  P_ a \not= \emptyset $,
then  
$ \bigcap _{ a \in A}  \overline{P}_ a \not= \emptyset $.

\item
Same as (2),
with the further assumption that $|P_a| \leq |a^<_B|$,
for every $a \in A$.

\item
For every  sequence 
$\{ x_H \mid H \in B \}$
of elements of  $X$,
it happens that
    $ \bigcap _{ a \in A} 
 \overline{ \{ x_H \mid H \in a^<_B \}}
  \not= \emptyset $.

\item
For every  sequence 
$\{ x_H \mid H \in B \}$
of elements of  $X$,
there is $x \in X$ 
such that, for 
 every neighborhood $U$ of $x$ in $X$,
and for every $a \in A$,
there is $H \in B$ such that $a \in H$
and $x_H \in U$.    
 
\item
For every  sequence 
$\{ Y_H \mid H \in B \}$
of nonempty subsets of  $X$,
it happens that
    $ \bigcap _{ a \in A} 
 \overline{ \bigcup \{ Y_H \mid H \in a^<_B \}}
  \not= \emptyset $.
  \item
 For every  sequence 
$\{ D_H \mid H \in B \}$
of nonempty closed subsets of  $X$,    $ \bigcap _{ a \in A} 
 \overline{ \bigcup \{ D_H \mid H \in a^<_B \}}
  \not= \emptyset $.

\item
For every  sequence 
$\{ O_H \mid H \in B \}$
of open proper subsets of  $X$,
 if, for every $a \in A$, we put  
$ Q_ a= \left( \bigcap  \{ O_H \mid H \in a^<_B \} \right)^\circ  $,
then $(Q_a)_{ a \in A}$  is not a cover of $X$.
\end{enumerate}
 \end{theorem}

\begin{proof}
(1) $\Rightarrow $  (2) 
Just take $C_ a=\overline{P}_ a$, for
$ a \in A$, and use the equivalent formulation 
of $ [ B, A]$-\brfrt compactness in terms of closed sets, as given in Fact \ref{factclosed}. 

(2) $\Rightarrow $  (3) is trivial.

(3) $\Rightarrow $  (4)
For $a \in A$,
put  
$ P_a= \{ x_H \mid H \in a^<_B \} $.
Thus 
$| P_a| \leq  | a^<_B | $.
Moreover, if
$a \in H \in B$, then $x_H \in P_a$,
hence $x_H \in \bigcap _{ a \in H}  P_ a $,
thus  $\bigcap _{ a \in H}  P_ a  \not= \emptyset $.
By applying (3), 
 $\bigcap _{ a \in A}  \overline{P_ a} = \bigcap _{ a \in A} 
 \overline{ \{ x_H \mid H \in a^<_B \}}
  \not= \emptyset $.

(5) is clearly a reformulation of (4), hence they are equivalent.

(4) trivially implies (6), since if, for every $H \in B$, we choose 
$x_H \in Y_H \not= \emptyset $, then  $ \bigcap _{ a \in A} 
 \overline{ \bigcup \{ Y_H \mid H \in a^<_B \}}
\supseteq
 \bigcap _{ a \in A} 
 \overline{ \{ x_H \mid H \in a^<_B \}}
 $, and this latter set is nonempty by (4).
  
(6) $\Rightarrow $   (7) is trivial, since (7) is a particular case of  (6).

(7) $\Rightarrow $   (1)
We shall use the equivalent formulation of $ [ B, A]$-\brfrt compactness
given by Fact \ref{factclosed}. 
Suppose that 
$( C _ a) _{ a \in A} $
are closed subsets of $X$ such that 
$  \bigcap _{ a \in H}  C_ a \not= \emptyset $,
for every $H \in B$.
For each $H \in B$,
put
$ D_H = \bigcap _{ a \in H}  C_ a $,
thus $C_a \supseteq D_H$,
whenever $a \in H$, hence,
for every $a \in A$,   
$ C_a \supseteq 
 \overline{ \bigcup \{ D_H \mid H \in a^<_B \}} $.
By (7),
$ \bigcap _{ a \in A} C_a \supseteq 
\bigcap _{ a \in A} 
 \overline{ \bigcup \{ D_H \mid H \in a^<_B \}}
  \not= \emptyset $.

(8) $\Leftrightarrow $  (7)
is immediate by taking complements.
 \end{proof}

Notice that Conditions (6) and (7) can be reformulated in a way similar to the reformulation (5) of (4). 
As we shall explain in detail in Section \ref{convcov},  Condition 
(5) can be seen as a statement that asserts the existence of some kind of 
accumulation point for the sequence  $\{ x_H \mid H \in B \}$.

\begin{remark} \labbel{classical}    
Some particular cases of Theorem \ref{e} are known, sometimes being classical results.

As we mentioned in the introduction, countable compactness is the particular case of 
Definition \ref{BAcpn} when $A$ is countable (without loss of generality we can take $A= \omega $), and $B$ is the set of all finite subsets of $ \omega$. It is easy to see that we can equivalently take $B= \{ [0, n) \mid n \in \omega \} $; this follows,
for example, from Remark \ref{noloss}. In a different context, 
a similar argument has been exploited in \cite{tapp2}; see in particular Remark 24 there.
Remark \ref{noloss} (and Fact \ref{fac}(2)(3))  have further interesting applications which shall be presented elsewhere.

 Recall that, for an infinite cardinal $\lambda$,
a topological space $X$ satisfies $\CAP_ \lambda  $ if and only if every subset
$Y \subseteq X $ with $| Y| = \lambda $ 
has a 
complete accumulation point $x$, that is, a point $x$ such that  
$| U \cap Y|= \lambda $, for every neighborhood $U$ of $x$.

For the above choice of $A= \omega $ and $B= \{ [0, n) \mid n \in \omega \} $,
the equivalence of (1) and (5) in Theorem \ref{e} shows that countable compactness
is equivalent to  $\CAP_ \omega $. This is because a sequence
$(x_H) _{H \in B} $ can be thought as a sequence 
$(x_n) _{n \in \omega } $, via the obvious correspondence between 
$B$ and $ \omega$. The astute reader 
will notice that the above argument (and Theorem \ref{e}, in general)
deals with sequences, while the definition 
of $\CAP_ \omega $ deals with  subsets;
that is, in the former case, repetitions are allowed, while in the latter case they are not allowed. 
 However, it is easy to see that, in the particular case at hand, 
the difference produces no substantial effect. 
See Remark \ref{seqvssubs} below and \cite[Section 3]{tproc2} for further details. 

Arguments similar to the above ones 
can  be carried over, with no essential change, for every
regular cardinal $\lambda$. In this case, we get that  $ [ \lambda  , \lambda ]$-\brfrt compactness
is equivalent to $\CAP_ \lambda  $. These results are very classical, and, indeed,
are immediate consequences of \cite[Section 9]{AU}.
For $\lambda$ singular, the  characterization
of $ [ \lambda , \lambda ]$-\brfrt compactness is not that neat.
The point is that, for $\lambda$ regular, a subset of $\lambda$ cofinal in $\lambda$ 
has necessarily cardinality $\lambda$; this is false when $\lambda$ 
is singular.
\end{remark}

We have discussed in some detail the 
 equivalence between   $\CAP_ \lambda  $ and $ [ \lambda  , \lambda ]$-\brfrt compactness, 
for $\lambda$ regular, since it might be seen as a prototype of all the
results proved in the present paper. In fact, we establish an interplay
between notions of compactness, on one hand, and  satisfaction of  accumulation properties, on the other hand. Such an
 interplay  holds in very general situations, sometimes rather  far removed from the above particular and nowadays standard  example.

Turning to the more general notion of   $ [ \mu, \lambda ]$-\brfrt compactness, the special case of the equivalence of (1) and (2)
in Theorem \ref{e} appears in  \cite[Theorem 1.1]{Ga}. See \cite[Lemma 5(b)]{Vfund}. 
For $ [ \mu , \lambda ]$-\brfrt compactness,
Conditions (1)-(4) in Theorem \ref{e} are the particular case of  
\cite
[Proposition 32 (1)-(4)]
{tapp2}, taking $\mathcal F$ to be the set of all singletons of $X$. 
In the particular case $\mu= \omega $, 
$ [ \omega   , \lambda ]$-\brfrt compactness is usually called \emph{initial $\lambda$-compactness}. In this case there are many more   characterizations: see \cite[Section 2]{St} and \cite{Vfund}. 
Some equivalences hold also for $\mu> \omega $, under 
additional assumptions.
See \cite[Theorem 2]{Vfund}. 

The equivalences in Theorem \ref{e} have been inspired by 
results from Caicedo \cite[Section 3]{C}, 
who implicitly  uses similar methods in order to deal with  $ [ \mu, \lambda ]$-\brfrt compactness.
In our opinion, Caicedo \cite{C} 
has provided an essentially new point of view about $ [ \mu, \lambda ]$-\brfrt compactness.
Apart from \cite{C}, it is difficult to track back which parts of Theorem \ref{e},
in this particular case, have appeared in some form or another in the literature.
This is
due to the hidden assumption, used by many authors, of the regularity of some of the cardinals  involved, or of some forms of the generalized continuum hypothesis. See \cite{Vfund}. 

Theorem \ref{e} is new in the particular case of $ [ \beta  , \alpha  ]$-\brfrt compactness, for $\beta$ and $\alpha$ ordinals. Since it  was our  leading motivation for working
on such matters, we state explicitly the equivalence of (1) and (4) in Theorem \ref{e}
for this special case.
We let $\mathcal P _ \beta ( \alpha )$ denote the set of all subsets of $ \alpha $ 
having order type $<\beta$. 
Notice that this notation is consistent with the case introduced before 
when $\alpha$ and $\beta$ are cardinals.

\begin{corollary} \labbel{alphabeta}  
Suppose that  $X$ is a topological space and $\alpha$ and $\beta$ are ordinals.
 Then the  following conditions are equivalent.
\begin{enumerate} 
\item
$X$ is $ [ \beta , \alpha ]$-\brfrt compact.
\item
For every  sequence 
$\{ x_z \mid z \in \mathcal P _ \beta ( \alpha ) \}$
of elements of  $X$, if, 
for $\gamma \in \alpha $, 
we put $P_ \gamma = \{ x_z \mid z \in \mathcal P _ \beta ( \alpha ) \text{ and  } \gamma \in z \}$,
then
    $ \bigcap _{ \gamma \in \alpha } 
 \overline{ P}_ \gamma 
  \not= \emptyset $.
\end{enumerate}
 \end{corollary}

As we mentioned in the introduction, also $D$-compactness turns out to be equivalent to a covering property
in the sense of Definition \ref{BAcpn}. More generally, many notions
of being closed under convergence, or under taking particular kinds of accumulation points are equivalent to a covering property, as we shall show in the next section. Theorem \ref{e} applies in each of
the above cases.

As a final remark in this section, let us mention that Condition (5) in Theorem 
\ref{e} suggests the following relativized notion of  a cluster point
of a net.

\begin{definition} \labbel{net}
Suppose that $( \Sigma, \leq)$ is a directed set, and $(x_ \sigma ) _{ \sigma \in \Sigma} $
is a net in a topological space $X$.
If $T \subseteq \Sigma$, we say that $x \in X$  is
a \emph{cluster point restricted to $T$} of the net 
$(x_ \sigma ) _{ \sigma \in \Sigma} $ 
if and only if for every $ \tau  \in T$ and every neighborhood $U$ of $x$,
there is $ \sigma \in \Sigma $ such that $ \sigma \geq \tau $ and
$x_ \sigma \in U$.
 \end{definition}   

In fact, if $\Sigma = B \subseteq \mathcal P(A)$, $\leq$ is inclusion, and we suppose that 
$B$ contains all singletons  of $\mathcal P(A)$,
then, in the terminology of Definition \ref{net}, 
 Condition \ref{e}(5) 
asserts that
every $ \Sigma $-indexed net 
 $(x_ \sigma ) _{ \sigma \in \Sigma} $ 
has some cluster point restricted to the set of
all singletons of $\mathcal P(A)$.

This might explain the difficulties in  finding
an equivalent formulation of
$ [ \mu, \lambda ]$-\brfrt compactness in terms
of cluster points of nets \cite{Vfund}. The condition 
in Definition \ref{net} is generally weaker than  
the request for a cluster point: 
the definition of a \emph{cluster point}
of a net is obtained from \ref{net} in the particular case when
$T= \Sigma $ (or, more generally, when $T$ is \emph{cofinal} in $\Sigma$,
that is, $T$ is such that, for every $ \sigma \in \Sigma $, there is
$ \sigma ' \in T$ such that $ \sigma \leq \sigma '$).

\section{Every notion of accumulation (and more) is a covering property} \labbel{convcov} 

An uncompromising way of defining a general notion of ``accumulation point''   
is simply to fix some index set $I$, and to prescribe exactly
which subsets of $I$ are allowed to be the (index sets
of) elements contained in the neighborhoods of some $x$---
 supposed to be an accumulation point  of some  $I$-indexed sequence.

Just to present the simplest nontrivial example, if $I$ is infinite, and we allow all subsets of  $I$
with cardinality  $|I|$, we get the  notion of
a complete accumulation point 
(for sequences all whose points are distinct).

 To state it  precisely, let us give the following definition.

\begin{definition} \labbel{seq}
Let $I$ be a set,  $E$ be a subset of $ \mathcal P(I)$,
and $\mathbf x =(x_i) _{i \in I} $ 
be an $I$-indexed sequence  of elements of some topological space $X$. 

If $U \subseteq X$, let  $I _{\mathbf x,U}=\{ i \in I \mid x_i \in U\}$.
We say that a point $x \in X$ is an
\emph{accumulation point in the sense of} $E$,
or simply an \emph{$E$-accumulation point},  
 of the sequence
$\mathbf x $ 
 if and only if
$I _{\mathbf x,U}\in E$, 
for every open neighborhood $U$ of $x$.

We say that $X$ satisfies the \emph{$E$-accumulation property}
if and only if every $I$-indexed sequence of elements of $X$
has some (not necessarily unique) accumulation point in the sense of $E$.
 \end{definition} 

\begin{remark} \labbel{minim}    
Trivially, if $E = \mathcal P(I)$, then every space satisfies the
$E$-\brfrt accumulation property. 
Under certain assumptions, we can get a smaller ``minimal'' $E$.

For every $I$-indexed sequence $\mathbf x$ of elements of $X$,
and every $x \in X$, there is a smallest set 
$E \subseteq  \mathcal P(I)$
such that $x$ is 
an $E$-accumulation point  
 of 
$\mathbf x$: just take 
$E = E _{\mathbf x , x} = \{  I _{\mathbf x,U} | U \text{ an open neighborhood of } x \} $.
Notice that $E _{\mathbf x , x}$ is closed 
under finite intersections and arbitrary unions.

More generally, if $ \Sigma $ 
is a set of $I$-indexed sequences  of elements of $X$
and, for every $\mathbf x \in \Sigma $, $Y_\mathbf x$
is a subset of $X$, then
 $ E=\bigcup  \{ E _{\mathbf x , x} \mid
\mathbf x \in \Sigma, \ x \in Y_\mathbf x \} $
is the smallest set $E$ such that
$x$ is 
an $E$-accumulation point  
 of 
$\mathbf x$,
for every $\mathbf x \in \Sigma $ and $ x \in Y_\mathbf x$.
In other words, if we fix in advance
some abstract relation of being an accumulation point of a sequence,
then there is a minimal $E$ which realizes this relation
(of course, in general, $E$ will realize many more instances of accumulation).
\end{remark}

\begin{remark} \labbel{seqvssubs}   
As we hinted before  Definition \ref{seq}, if $I$ is infinite, and 
$E$ is the set of all subsets of $I$  of cardinality $|I|$,
then the notion of an $E$-accumulation point corresponds to 
that of a complete accumulation point.
There is a
technical difference that should be mentioned:
 here we are dealing with sequences, rather than subsets.
In order to get the standard definition of a complete
accumulation point, we should require that
all the elements of the sequence are distinct, otherwise some
differences might occur. 
However, if $|I|$ is a regular cardinal, then a topological space satisfies 
$\CAP _{|I|}  $ if and only if it satisfies the $E$-accumulation property,
for the above $E$.

The whole matter has been discussed in detail
in \cite[Section 3]{tproc2}, see in particular Remark 3.2 and Proposition 3.3 there,
taking $\mathcal F$ to be the set of all singletons of $X$. We believe that, in general, dealing with sequences 
is the most natural way; for sure, it is the best way for our
purposes here.
\end{remark} 

\begin{remark} \labbel{resemb}  
Definition \ref{seq} has some resemblance with 
the notion of filter convergence. However,  
we are not asking  $E$ to be necessarily a filter. This is because
 we want to include notions
of accumulation and since, for example,  in the case
of  complete accumulation points the corresponding
$E$ is not  closed under intersection.
Indeed, the intersection  of two subsets of  $I$
of cardinality  $|I|$ may have cardinality strictly smaller than $|I|$.

Of course, given some \emph{fixed} sequence  $(x_i) _{i \in I} $ and some \emph{fixed} 
element $x \in X$, the topological relations between  $(x_i) _{i \in I} $
and $x$ are completely determined by the (possibly improper) filter $F$  
generated by the sets $\{ i \in I \mid x_i \in U\}$, $U$ varying among 
the neighborhoods of $x$ in $X$. However, as the example of complete accumulation points shows, if we allow  
 $x$  vary, we get a more general (and useful) notion by considering 
an arbitrary subset $E$, rather than just a filter.

In this connection, however, see also Remark \ref{bastfiltr}. 
\end{remark}  

Definition \ref{seq}  incorporates essentially all possible
notions of ``accumulation''. It captures also many
notions of convergence.
For example, a sequence $(x_n) _{n \in \omega } $ 
converges to $x$ if and only if, for every neighborhood $U$
of $X$, the set $ \omega \setminus \{ n \in \omega \mid x_n \in U  \}$ is finite.     
In this case, $I= \omega $ and $E$ consists of the  cofinite subsets of $ \omega$.
In a similar way, we can deal with convergence of transfinite sequences. Actually,
even net convergence is a particular case of Definition \ref{seq}. If
$( \Sigma, \leq)$ is the directed set on which the net is built, then
the net converges to $x$ if and only if 
$x$ is an $E$-accumulation point in the sense of 
Definition \ref{seq} for the following choice of $E$.
Take $I= \Sigma$ and let
$E $ be the set of all subsets of $I$
which contain at least one set of the form $ \sigma^<$,
where, for $\sigma \in \Sigma$, we put
$\sigma^< = \{  \sigma ' \in \Sigma \mid  \sigma \leq \sigma '\} $.
Of course, this is the usual argument showing that net convergence can be seen as 
an instance of filter convergence.    

Definition \ref{seq} is more general. If,
for a net as above,
 we take $E$ to be the set of all subsets of $ \Sigma$ 
which are cofinal in $ \Sigma$, then an $E$-accumulation point corresponds to 
a cluster point of the net. Also the notion of a restricted cluster point,
as introduced in Definition \ref{net}, can be expressed in terms
of $E$-accumulation, for some appropriate $E$. 

If $E =D$ is an   ultrafilter    over $ I$, then 
the existence of an $E$-accumulation point 
corresponds exactly to $D$-convergence.

It is rather astonishing that such a bunch of  disparate notions
turn out to be each equivalent to some covering property in the sense of
Definition \ref{BAcpn}, as we shall show in Corollary \ref{rcor} below.

Before embarking in the proof, we notice that also the converse holds, that is, every
covering property is equivalent to some accumulation property. This is
simply a reformulation, in terms of $E$-accumulation, of
the equivalence  (1) $\Leftrightarrow $  (5) in Theorem \ref{e}.

\begin{corollary} \labbel{covisacc}
Suppose that $X$ is a topological space, 
$A$ is a set, $B \subseteq \mathcal P(A)$, and put 
$I=B$ and
 $E= \{ Z \subseteq B \mid \text{for every } a \in A
\text{ there is }
 H \in Z \text{ such that } a \in H \}
= \{ Z \subseteq B \mid \bigcup Z =A\}  $.
 Then the following conditions are equivalent.
  \begin{enumerate}    
\item  
$X$ is $ [ B , A]$-\brfrt compact.
\item
$X$ satisfies the $E$-accumulation property.
  \end{enumerate}  
\end{corollary} 

\begin{example} \labbel{0n} 
As in Remark \ref{classical},
if $A= \lambda $ is a regular infinite cardinal, and
$B= \{ [0,  \alpha ) \mid \alpha < \lambda  \}$,
then the $E$ given by Corollary \ref{covisacc}
consists of all subsets of $B$ of cardinality $\lambda$.
In this particular case, Corollary \ref{covisacc}
amounts exactly to the equivalence of $ [ \lambda  , \lambda ]$-\brfrt compactness
and $\CAP_ \lambda  $.
\end{example}   

\begin{example} \labbel{discr} 
As another  simple example, suppose that $A$ is any set, and let
$B= \{ A \setminus \{ a \} \mid a \in A\}$.  
For this choice of $B$, a topological space $X$ is 
 $ [ B , A]$-\brfrt compact if and only if $X$ has no irreducible
cover of cardinality $| A|$.
The $E$ given by Corollary \ref{covisacc} in this situation
is the set of all subsets of $B$ which contain
at least two elements from $B$.  
In this case, the failure of the $E$-accumulation property
means that there exists an $|A|$-indexed sequence of  elements of $X$  
 such that every element of $X$
has a neighborhood intersecting at most one element from the sequence.
If $X$ is $T_1$, this is equivalent to saying that
$X$ has a discrete closed subset of cardinality $|A|$.

In conclusion, in this particular case, Corollary \ref{covisacc}
shows that a $T_1$ topological space has an irreducible cover
of cardinality $\lambda$ if and only if it has a discrete closed subset of cardinality
$\lambda$.
This is a classical result, implicit in the proof of \cite[Theorem 2.1]{AD}.
\end{example}

Now we are going to prove the promised   converse of Corollary \ref{covisacc}, namely, that
every $E$-accumulation property in the sense of
Definition \ref{seq} is equivalent to some covering property,
under the reasonable hypothesis that $E$ is closed under taking supersets.

\begin{definition} \labbel{ast}
If $I$ is a set, and $E \subseteq  \mathcal P(I)$,
we let 
$E^+= \{a\subseteq I \mid a \cap e \not= \emptyset \text{, for every } e \in E \} $.

We say that  $E \subseteq  \mathcal P(I)$ is \emph{closed under
supersets} if and only if, whenever   $e \in E$ and $e \subseteq f \subseteq I$, 
then $f \in E$ (this is half the definition of a \emph{filter}: if $E$ is also closed
under finite intersections, then it is a filter).  
\end{definition}   

Trivially, for every $E$,
we have that $E^+$ is closed under supersets. 
Moreover, it is easy to see that $E ^{++}=E $
if and only if $E$ is closed under supersets. 
%Notice also that, in case $E$ is a filter, then $E^+$
%is the complement in  $\mathcal P(I)$ of the dual ideal 
%of $E$. 
%Symmetrically, if $\mathcal P(I) \setminus E$ is an ideal,
%then $E^+$ is the filter dual to this ideal. 
%In particular
Notice that if $E$ is a filter, then $E$ is an ultrafilter if and only if
$E=E^+$.

If $A \subseteq \mathcal P(I)$, then, for every $i \in I$, we put
$i^<_A = \{ a \in A \mid i \in a\} $.

We can now state the main result of this section.

\begin{theorem} \labbel{r}
Suppose that $X$ is a topological space, $I$ is a set,
 $A \subseteq  \mathcal P(I)$, and 
let $E=A^+$.
Then the following conditions are equivalent.
  \begin{enumerate}   
 \item   
$X$ satisfies the $E$-accumulation property.
\item
For every open cover $(O_a) _{a \in A} $
of $X$, there is $i \in I$
such that  
 $(O_a) _{i \in a \in A} $ is a cover of $X$.
\item 
$X$ is $ [ B , A]$-\brfrt compact, for $B= \{i^<_A \mid i \in I\} $.
  \item  
For every sequence
$(x_i) _{i \in I} $  of elements of $X$, 
if, for each $a \in A$,  we put 
$C_a = \overline{\{ x_i \mid i \in a \}}   $, 
then 
$\bigcap _{a \in A}C_a \not= \emptyset  $.  
\end{enumerate} 
 \end{theorem}

If  $E \subseteq  \mathcal P(I) $ is closed under supersets,
then $E=E ^{++}$,  hence, by taking 
  $A=E^+$, 
we have $E=E ^{++}=A^+$.
Thus we get from Theorem \ref{r} (1)$\Leftrightarrow $(3) that, for every 
$E$  closed under supersets,
 the $E$-accumulation property is equivalent to some compactness 
property in the sense of Definition \ref{BAcpn}. 

\begin{corollary} \labbel{rcor}
For every 
$E \subseteq  \mathcal P(I) $ such that $E$ is  closed under supersets,
there  are $A \subseteq  \mathcal P(I) $ and $B \subseteq  \mathcal P(A) $
such that, for every topological space,
 the $E$-accumulation property
is equivalent to $ [ B , A]$-\brfrt compactness.
 \end{corollary}

Before proving Theorem \ref{r},
we give a stronger local version for the equivalence
of Conditions (1) and (4).

\begin{proposition} \labbel{ae}
Suppose that $X$ is a topological space, $x \in X$, $I$ is a set,
and $(x_i) _{i \in I} $ is a sequence of elements of $X$.
Suppose that  
 $A \subseteq  \mathcal P(I)$, $E=A^+$,
and, 
for  $a \in A$,  put 
$C_a = \overline{\{ x_i \mid i \in a \}}   $.

Then the following conditions are equivalent. 
  \begin{enumerate}   
 \item
x is an $E$-accumulation point   of 
$(x_i) _{i \in I} $. 
  \item  
$ x \in \bigcap _{a \in A}C_a $.  
\end{enumerate}
 \end{proposition} 

 \begin{proof} 
If (1) holds, then,
for every open neighborhood $U$
of $x$, the set $e_U=\{ i \in I \mid x_i \in U  \} $ belongs to $ E$.
We are going to show that 
 $ x \in \bigcap _{a \in A}C_a   $.

Hence, suppose that $a \in A$.
For every open neighborhood $U$ of $x$, 
$a \cap e_U \not= \emptyset $, by the first statement, and the definition of $E$.
This means that there is $i \in I$ such that   $i \in a \cap e_U$,
that is, $x_i \in C_a \cap U$, hence  $ C_a \cap U \not= \emptyset $.
Since $C_a$
is closed, and the above inequality 
 holds for every open neighborhood $U$ of $x$, then $x \in C_a$.      
Since this holds for every $a \in A$, we have  
 $ x \in \bigcap _{a \in A}C_a   $.

Now assume that (2) holds. 
Suppose that $U$ is a neighborhood of $x$, and let
$e=\{i \in I \mid x_i \in U\}$. We have to show that
$e \in E= A^+$, that is, $e \cap a \not= \emptyset  $,
for every $a \in A$. Let us fix $a \in A$.      
By (2), $x \in C_a$ and, by the definition of $C_a$,
there is $i \in a$ such that $x_i \in U$.   
By the definition of $e$, $i \in e$, thus
$i \in e \cap a \not= \emptyset $.
Since this argument works for every neighborhood $U$ of $x$,
we have proved (1). 
\end{proof} 

The particular case of Proposition \ref{ae} in which
$x$ is a cluster point of some net is 
Exercise 1.6.A in \cite{E}.
Cf. also \cite[IV]{Cho}.

\begin{proof}[Proof of Theorem \ref{r}]
(2) $\Leftrightarrow $  (3) is immediate from the definitions.

(3) $\Leftrightarrow $  (4) is a particular case of Theorem \ref{e} (1) $\Leftrightarrow $  (4). Indeed, in the situation
at hand, members of $B$ have the form $H=i^<_A$, for 
$i \in I$. For such an $H$, we have that $H \in a^<_B$ 
if and only if $a \in H=i^<_A $ if and only if $i \in a$, thus
Condition (4) in Theorem \ref{e} reads exactly as 
Condition (4) in Theorem \ref{r}. 

(1) $\Leftrightarrow $  (4) is immediate from Proposition \ref{ae}.

Alternatively, the proof of  \ref{r} can be completed avoiding the use of  
 Proposition \ref{ae}, and using Corollary \ref{covisacc} instead.
 Indeed, modulo the obvious correspondence 
between $I$ and $B= \{i^<_A \mid i \in I\} $,
 the $E$ in \ref{covisacc} corresponds exactly to 
the $E$ in Theorem \ref{r}. To check this,
let $I'=B$ and, for $e \subseteq I$, let 
$e' \subseteq I'$ be defined by 
$e'= \{ i^<_A \mid i \in e\}$.
Applying Corollary \ref{covisacc}
to $I'$, the resulting $E'$ turns out to be
equal to 
$\{ e' \subseteq  I' \mid 
\text{for every $a \in A$,  there is $i \in I$ such that  
$i^<_A \in e'$ and $i\in A$} \}=
\{ e' \subseteq I' \mid e \cap a  \not= \emptyset,   
\text{ for every $a \in A$}  \}= 
\{ e'| e \in E\}$.
Corollary \ref{covisacc} thus shows that 
$ [ B , A]$-\brfrt compactness
is equivalent to the $E'$-accumulation property,
 which, through the above mentioned correspondence,  is trivially 
equivalent
to the $E$-accumulation property.
\end{proof}

\begin{remark} \labbel{d} 
If $D$ is an ultrafilter over $I$, 
then, 
by taking $A=D$
in Theorem \ref{r}, the equivalence of (1) and (2) furnishes 
a proof of Proposition \ref{DAB}, since, for $D$ an ultrafilter, 
we have that $D^+=D$.  
\end{remark}   

In \cite[Proposition 17]{tapp2} we also proved a characterization of 
$D$-\brfrt pseudocompactness 
 analogous to 
Proposition \ref{DAB}.
The methods of Sections \ref{top} and \ref{convcov} do apply also in case of notions related
to pseudocompactness. We shall devote the next section to this
endeavor.
Before proceeding, we show that Theorem \ref{r} furnishes a characterization of
weak $M$-compactness. 

\begin{definition} \labbel{weakly}    
If $M$ is a set of ultrafilters over some set $I$, a topological space
is said to be \emph{weakly $M$-compact} if and only if,
for every sequence $(x_i ) _{i \in I}$ of elements of $X$,
there is $x \in X$ such that, for every neighborhood $U$ of $x$,
there is $D \in M$ such that  $\{i \in I \mid x_i \in U\}$. See \cite{GFwc} for more information,
credits, references and a characterization.
 In the terminology of Definition \ref{seq},  
$X$ is  weakly $M$-compact if and only if it satisfies the 
$E$-accumulation property, for $E= \bigcup_{D \in M} D$. 
\end{definition}

\begin{corollary} \labbel{wc}
Suppose that $X$ is a topological space, $M$ is a set of ultrafilters over $I$, and let 
$F = \bigcap _{D \in M} D$.  Then the following conditions are equivalent. 
  \begin{enumerate}    
\item 
$X$ is weakly $M$-compact.
\item
For every open cover $(O_Z) _{Z \in F} $ of $X$,
there is some $i \in I$ such that  
$(O_Z) _{i \in Z \in F} $
is a cover of $X$.
  \end{enumerate} 
 \end{corollary}

\begin{proof} 
By Theorem \ref{r},
taking $A=F$, and noticing that 
$E=A^+=  \bigcup_{D \in M} D$.  
\end{proof}

\section{Pseudocompactness and the like} \labbel{pseudocpn} 

Definitions \ref{BAcpn} and \ref{seq} 
 can be generalized in the setting
presented in \cite{tproc2,tapp2}; in 
particular, in such a way that incorporates  pseudocompact\brfrt -\brfrt like notions.

Let us  fix a family $\mathcal F$ of subsets of a topological space $X$.
The most interesting case will be
when $\mathcal F =\mathcal O$  is the family
of all the nonempty open sets of $X$. At first reading,
the reader might want to consider this particular case only.

We relativize Definitions \ref{BAcpn} and \ref{seq} to $\mathcal F$.
The notion of $ [ B , A ]$-compactness is modified
by replacing the conclusion with the 
requirement that  the union of the elements of an
appropriate  subsequence intersects ever member of $\mathcal F$.
As far as notions of accumulation are concerned,
 instead of considering accumulation points of elements,
we shall now consider limit points of sequences of elements of $\mathcal F$.

The two most significant cases are when $\mathcal F$ is the family
of all singletons of $X$, in which case we get back the definitions and results
of Sections \ref{top} and \ref{convcov}, and, as we mentioned,  when
 $\mathcal F =\mathcal O$ is the family
of all the nonempty open sets of $X$, in which case we get notions and results
related to pseudocompactness or variants of pseudocompactness.

\begin{definition} \labbel{ABO}
If $A$ is a set, $B \subseteq \mathcal P(A)$, $X$ is a topological space,
and $\mathcal F$ is a family of subsets of  $X$, we say that $X$  is
$\mathcal F$-$ [ B , A ]$-\emph{compact}
if and only if one of the following equivalent conditions holds.
  \begin{enumerate}    
\item   
For every open cover $( O _ a) _{ a \in A   } $
 of $X$, 
there is  $H \in B  $ such that
$\bigcup_{a \in H} O_a$ intersects every member
of $\mathcal F$ (that is, 
for every  $F\in \mathcal F$, there is $a \in H$ such that   
$  O_ a \cap F\not= \emptyset $).
\item   
For every sequence  $( C _ a) _{ a \in A   } $
 of closed subsets of $X$, if,
for every $H \in B  $, 
there exists $F\in \mathcal F$ such that   
$ \bigcap _{ a \in H}  C_ a \supseteq F$,
then  
$ \bigcap _{ a \in A}  C_ a \not= \emptyset $.
\end{enumerate} 
\end{definition}

The equivalence of the above conditions  is trivial, by taking complements.

Notice that, in the particular case when 
$\mathcal F =\mathcal O$,
 the conclusion in  Definition \ref{ABO} (1)  asserts that
$ \bigcup _{ a \in A}  O_ a $ is dense in $X$.

\begin{definition} \labbel{seqO}
Let $I$ be a set,  $E$ be a subset of $ \mathcal P(I)$,
and $(F_i) _{i \in I} $ 
be an $I$-indexed sequence  of subsets of some topological space $X$. 

We say that a point $x \in X$ is a
\emph{limit point in the sense of} $E$,
or simply an \emph{$E$-limit point},  
 of the sequence
$(F_i) _{i \in I} $ 
 if and only if,  
for every open neighborhood $U$
of $x$, the set $\{ i \in I \mid F_i \cap U \not= \emptyset  \} $ belongs to $ E$.

If $\mathcal F$ is a family
of subsets of $X$, we say that $X$ satisfies the \emph{$\mathcal F$-$E$-accumulation property}
if and only if every $I$-indexed sequence of elements of $\mathcal F$ 
has some  limit point in the sense of $E$.
 \end{definition}

In the particular case when $\mathcal F$ is the family of all singletons
 of $X$ Definitions \ref{ABO} and \ref{seqO} reduce to Definitions \ref{BAcpn}
and \ref{seq}, respectively.   

As in Remark \ref{minim},    
 if $E = \mathcal P(I)$, then every space satisfies the
$\mathcal F$-$E$-\brfrt accumulation property, for every $\mathcal F$. 

More generally, for every  sequence $(F_i) _{i \in I} $ of subsets of $X$,
and every $x \in X$, 
there is a smallest set 
$E \subseteq  \mathcal P(I)$
such that $x$ is 
an  $E$-limit point  
 of 
$(F_i) _{i \in I} $: just take 
$E = \{  I _{U} | U \text{ an open neighborhood of } x \} $, where
$I_U = \{ i \in I \mid F_i \cap U  \not= \emptyset   \}$. 
In the same way, and exactly as in Remark \ref{minim}, for every family of 
$I$-indexed sequences,
and  respective families of elements of $X$,  
there is the smallest $E$ such that each element in the family 
is a limit point of the corresponding sequence.

\begin{remark} \labbel{equiv}
If $\mathcal F$ is a family of subsets of some topological space $X$,
let $ \overline{ \mathcal F}$ denote the set of all closures of
elements of $\mathcal F$.

If $\mathcal G$ is another family of subsets of $X$,
let us write 
$\mathcal F \rhd \mathcal G$ to mean that,
 for every $F \in \mathcal F$, there is $G \in \mathcal G$
such that $F \supseteq G$.
We write $\mathcal F \equiv \mathcal G$  to mean that 
both
$\mathcal F \rhd \mathcal G$
and $\mathcal G \rhd \mathcal F$.

It is trivial to see that, in Definitions \ref{ABO} and \ref{seqO},
as well as in the theorems below,
we get equivalent conditions if we replace $\mathcal F$ either by   
$ \overline{ \mathcal F}$, or by $\mathcal G$, in case 
$\mathcal F \equiv \mathcal G$ 
(in this latter case, as far as Definition \ref{seqO} is concerned, 
the condition turns out to be equivalent provided we assume that $E$
is closed under supersets).

 In particular, when $\mathcal F =\mathcal O$, we get equivalent definitions
and results if we replace $\mathcal O$  by either    
  \begin{enumerate}    
\item
the set $\mathcal B$  of the nonempty 
elements of some fixed base
of $X$, or 
\item
the set $ \overline{\mathcal O}$  of 
all nonempty regular closed subsets of $X$, or
\item  
the set $ \overline{\mathcal B}$  of the closures of the nonempty 
elements of some base
of $X$, or
\item
the set $ \mathcal R$  of 
all nonempty regular open subsets of $X$
(since $ \overline{\mathcal R}= \overline{\mathcal O}$).
 \end{enumerate} 
   \end{remark}

The connection of  Definitions \ref{ABO} and \ref{seqO}
with pseudocompactness goes as follows.
A Tychonoff space $X$ is
pseudocompact if and only if 
every countable open cover of $X$  has a finite subcollection
 whose union is dense in $X$. This is Condition
(C$_5$) in \cite{sted7}, and corresponds to the particular case
$A= \omega $, $B= \mathcal P_ \omega ( \omega )$ 
  of $\mathcal O$-$ [ B , A ]$-compactness,
in the sense of Definition \ref{ABO}. 

As another characterization of pseudocompactness,
Glicksberg \cite{Gl}  proved that
 a Tychonoff space $X$ is pseudocompact 
if and only if the following condition holds:

(*) for every  sequence of nonempty open sets of $X$,
there is some
 point $x \in X$ such that  each neighborhood of $x$ 
 intersects
infinitely many elements of the sequence.

This  corresponds to the particular case
of Definition \ref{seqO} in which
$\mathcal F =\mathcal O$,
$I = \omega $ and $E $ equals 
the set of all infinite subsets of $ \omega$.    
Actually, as a very particular case of Theorem \ref{eO}
 (1) $\Leftrightarrow $  (5) below,
and arguing as in Remark \ref{classical}, 
we get another proof of Glicksberg result, in the sense that we get a proof
that (*) and (C$_5$) above are equivalent, for every topological space 
(no separation axiom assumed).

The situation is entirely parallel to the characterization of
countable compactness, which is equivalent to $\CAP _ \omega $, 
as discussed in detail  in Remark \ref{classical}. 
Indeed, conditions analogous to 
(*) and (C$_5$) above are still equivalent when
$ \omega$ is replaced by any infinite regular cardinal; 
see \cite[Theorem 4.4]{tproc2} for exact statements.
This kind of analogies, together with many generalizations, had been the main
theme of \cite{tproc2,tapp2}. 
In the present paper we show that
such analogies can be carried over much further.

The connections between covering properties 
and general accumulation properties, as 
described in Section \ref{convcov}, 
do hold even in the extended setting we are now considering.
In other words, the relationships between the properties
introduced in Definitions \ref{BAcpn}  
and \ref{seq} are exactly the same as the relationships between
the properties introduced in  
Definitions \ref{ABO}  
and \ref{seqO}.
This shall be stated in Theorem \ref{rO}.
 
We first state the result analogous to Theorem \ref{e} (and Corollary \ref{covisacc}).

\begin{theorem} \labbel{eO} 
Suppose that $A$ is a set, $B \subseteq \mathcal P(A)$,  $X$ is a topological space,
and $\mathcal F$ is a family of subsets of $X$. Then the following conditions are equivalent.
\begin{enumerate} 
\item
$X$ is $\mathcal F$-$ [ B, A]$-\brfrt compact.
\item
For every sequence  $( P _ a) _{ a \in A   } $
 of subsets of $X$, if,
for every $H \in B  $, 
there exists $F\in \mathcal F$ such that   
$ \bigcap _{ a \in H}  P_ a \supseteq F$,
then  
$ \bigcap _{ a \in A}  \overline{P}_ a \not= \emptyset $.
\item
Same as (2),
with the further assumption that,
for every $a \in A$,
$P_a  $ is the union of $ \leq \kappa_a $-many elements of
$\mathcal F$, where $\kappa_a= |a^<_B|$.
\item
For every  sequence 
$\{ F_H \mid H \in B \}$
of elements of $\mathcal F$,
  it happens that   $ \bigcap _{ a \in A} 
 \overline{ \bigcup \{ F_H \mid H \in a^<_B \}}
  \not= \emptyset $.
\item
For every  sequence 
$\{ F_H \mid H \in B \}$
of elements of  $\mathcal F$,
there is $x \in X$ 
such that, for every neighborhood $U$ of $x$ in $X$,
and for every $a \in A$,
there is $H \in B$ such that $a \in H$
and $F_H \cap U \not= \emptyset $.    
\item
For every  sequence 
$\{ Y_H \mid H \in B \}$
of subsets of  $X$
such that each $Y_H$ contains some 
$F_H \in \mathcal F$, 
%it happens that
    $ \bigcap _{ a \in A} 
 \overline{ \bigcup \{ Y_H \mid H \in a^<_B \}}
  \not= \emptyset $.
\item
 For every  sequence 
$\{ D_H \mid H \in B \}$
of  closed subsets of  $X$
such that each $D_H$ contains some 
$F_H \in \mathcal F$, 
it happens that  $ \bigcap _{ a \in A} 
 \overline{ \bigcup \{ D_H \mid H \in a^<_B \}}
  \not= \emptyset $.
\item
For every  sequence 
$\{ O_H \mid H \in B \}$
of open  subsets of  $X$
   such that, for each $H \in B$,  there is
$F_H \in \mathcal F$ disjoint from
$O_H$, if, for every $a \in A$, we put  
$ Q_ a= \left( \bigcap  \{ O_H \mid H \in a^<_B \} \right)^\circ  $,
then $(Q_a)_{ a \in A}$  is not a cover of $X$.
\item 
$X$ satisfies the $\mathcal F$-$E$-accumulation property, for
$I=B$ and
 $E= \{ Z \subseteq B \mid \text{for every } a \in A
\text{ there is }
 H \in Z \text{ such that } a \in H \} $.
\end{enumerate}

In each case, we get equivalent conditions by replacing $\mathcal F$  with either
 $ \overline{\mathcal F}$, or 
$\mathcal G$, in case 
$\mathcal F \equiv \mathcal G$.
 \end{theorem}  

\begin{proof}
The proof is similar to the proof of Theorem \ref{e}.
Cf.~also parts of the proof of \cite[Proposition 6]{tapp2}.

It is not  obvious that we get 
equivalent statements for all conditions, when $\mathcal F$  is replaced by  
 $ \overline{\mathcal F}$, or by 
$\mathcal G$, when
$\mathcal F \equiv \mathcal G$.
However, this is true for, say, Condition (1),
and the proof of the equivalences of (1) - (9)
works for an arbitrary family.
\end{proof}
 
As a simple example of the equivalence of (1) and (9),
and arguing as in Remark \ref{discr},
 a topological  space $X$ has an open cover
of cardinality $\lambda$ with no proper dense subfamily
if and only if 
$X$ contains a discrete family of $\lambda$  open sets.

We now state the results corresponding to those in Section
\ref{convcov}. There is no essential difference in proofs. 

\begin{theorem} \labbel{rO}
Suppose that $X$ is a topological space,
 $\mathcal F$ is a family of subsets of $X$, $I$ is a set, 
 $A \subseteq  \mathcal P(I)$ and  $E=A^+$.

Then the following conditions are equivalent.
  \begin{enumerate}   
 \item   
$X$ satisfies the $\mathcal F$-$E$-accumulation property.
\item
For every sequence  $( C _ a) _{ a \in A   } $
 of closed subsets of $X$, if,
for every $i \in I  $, 
there exists $F\in \mathcal F$ such that   
$ \bigcap _{ i \in a \in A }  C_ a \supseteq F$,
then  
$ \bigcap _{ a \in A}  C_ a \not= \emptyset $.
\item 
$X$ is $\mathcal F$-$ [ B , A]$-\brfrt compact, where $B= \{i^<_A \mid i \in I\} $.
  \item  
For every sequence
$(F_i) _{i \in I} $  of elements in $\mathcal F$, 
if, for each $a \in A$,  we put 
$C_a = \overline{ \bigcup _{i \in a}  F_i }   $, 
then 
$\bigcap _{a \in A}C_a \not= \emptyset  $.  
\end{enumerate} 

In each case, we get equivalent conditions by replacing $\mathcal F$  with either
 $ \overline{\mathcal F}$, or 
$\mathcal G$, in case 
$\mathcal F \equiv \mathcal G$.
 \end{theorem} 

We state explicitly also the analogue of Proposition \ref{ae},
 since it does not follow formally
from Theorem \ref{rO}.   

\begin{proposition} \labbel{aeO}
Suppose that $X$ is a topological space, $x \in X$, $I$ is a set,
and $(F_i) _{i \in I} $ is a sequence of subsets of $X$.
Suppose that $A \subseteq  \mathcal P(I)$, $E=A^+$, 
and, 
for  $a \in A$,  put 
$C_a = \overline{ \bigcup _{i  \in a} F_i }   $.

Then the following conditions are equivalent. 
  \begin{enumerate}   
 \item
x is an $E$-limit point   of 
$(F_i) _{i \in I} $. 
  \item  
$ x \in \bigcap _{a \in A}C_a $.  
\end{enumerate}
 \end{proposition} 

A version of  Proposition \ref{aeO}
appears in \cite[IV]{Cho}, using different terminology and notations,
and possibly with a misprint.  
Proposition \ref{aeO} appears to be slightly more general,
since $E$ does not necessarily become a filter
(cf. Remark \ref{resemb}).

As an example, Theorem \ref{rO} can be applied to notions related
to ultrafilter convergence, in particular, to $D$-pseudocompactness. 

\begin{definition} \labbel{uf}
Let $D$ be an ultrafilter over some set $I$, $X$ be a topological space,
and $\mathcal F$ be a family of subsets of $X$.

We say \cite[Definition 2.1]{tproc2} that $X$ is  $\mathcal F$-$D$-\emph{compact} if and only if every sequence
$(F_i)_{i \in I}$ 
of members of $\mathcal F$ 
has some $D$-limit point in $X$.

In case $\mathcal F$ is the set of all singletons of $X$, we get back
the notion of $D$-compactness. In case $\mathcal F=\mathcal O$ we get the notion of
$D$-pseudocompactness, as introduced in  \cite{GS,GF}.
 \end{definition}   

\begin{cor} \label{equivodcpnf} 
\cite[Proposition 33]{tapp2}  
Suppose that $X$ is a topological space, $\mathcal F$ 
is a family of subsets of $X$, and 
$D$ is an ultrafilter over some set $I$.
Then the following are equivalent.
\begin{enumerate} 
\item 
$X$ is $\mathcal F$-$D$-\brfrt compact.

\item
For every sequence 
$\{ F_i \mid i \in I \}$
of members of $\mathcal F$, if, for 
$Z \in D$, we put
$C_Z= \overline{ \bigcup _{i \in Z} F_i} $, then
we have that 
$\bigcap _{Z \in D} C_Z \not= \emptyset  $.

\item 
Whenever $(C_Z) _{Z \in D} $ is a  sequence  of closed sets of $X$
with the property that, for every $i \in I$, there exists some $F \in \mathcal F$  
such that $\bigcap _{i\in Z} C_Z \supseteq F$, then  
$\bigcap _{Z \in D} C_Z \not= \emptyset  $.

\item 
For every open cover $(O_Z) _{Z \in D} $ of $X$,
there is some $i \in I$ such that  
$F \cap \bigcup _{i\in Z} O_Z   \not= \emptyset $,
for every $F \in \mathcal F$.
\end{enumerate} 
\end{cor}

In the particular case $\mathcal F =\mathcal O$, 
Corollary  \ref{equivodcpnf} provides a characterization
of $D$-pseudocompactness parallel to the characterization
of $D$-compactness given in  Proposition \ref{DAB}.  
This characterization
of $D$-pseudocompactness
had been explicitly stated with a direct proof in \cite[Proposition 17]{tapp2}.
Also Corollary \ref{wc} can be generalized without difficulty.
We leave this to the reader. 

Of course, 
all the results of Sections \ref{top} and \ref{convcov},
in particular,   
Theorems \ref{e} and \ref{r},
could be obtained as particular cases of
the results in the present section,
by taking $\mathcal F$ to be the set of all 
singletons of $X$.
In principle, we could have first proved Theorem \ref{eO} and \ref{rO}, 
and then obtain Theorems \ref{e} and \ref{r} as corollaries. 
We have chosen the other way for easiness of presentation, and since
already Sections \ref{top} and \ref{convcov} appear 
to be abstract enough.
Probably, there are more readers (if any at all!) interested in 
Theorems  \ref{e} and \ref{r} rather than in Theorems  \ref{eO} and \ref{rO} 
in such a generality. 

However, the particular case $\mathcal F =\mathcal O$  
in the results of the present section appears to be of interest.
We stated the results in the general $\mathcal F$-dependent form for three reasons.
First, to point out that, even if it is possible 
that the results are particularly interesting only in 
the case $\mathcal F =\mathcal O$, nevertheless   almost nowhere we made use of the specific form of the members of $\mathcal O$.
Second, since it is not always trivial that we can
equivalently replace $\mathcal O$  with anyone
 of the families (1) - (4) of Remark \ref{equiv}.
The general form of our statements thus provides many equivalences at the same time.
The third reason for stating the theorems in 
the $\mathcal F$-form is to make clear that 
there is absolutely no difference, in the proofs and in the arguments,
with the case dealt in the preceding sections, that is, when dealing with 
sequences of points, rather than general subsets.
In fact, the statements  of Theorems  \ref{eO} and \ref{rO} unify the two cases.
This is similar to what we have done in 
\cite{tproc2}; indeed, some results of \cite{tproc2}
can be obtained as corollaries of results proved here. 

Of course, the possibility is  left open for interesting applications
of Theorems  \ref{eO} and \ref{rO} in other cases, besides the cases of 
singletons and of nonempty open sets.

\section{Notions related to sequential compactness} \labbel{sequential} 

Sequential compactness is not a particular case of Definition \ref{seq}.
However, Definition \ref{seq} can be modified
in order to include also notions such as sequential compactness. 
The results in Sections \ref{convcov} and \ref{pseudocpn} generalize even to this situation.

 \begin{definition} \labbel{sequent}
 Suppose that $I$ is a set,  $\mathcal E$ is a set of subsets of
$ \mathcal P (I)$, and $X$ is a topological space. 
  \begin{enumerate}    
\item  
If $(x_i) _{i \in I} $ is
a sequence of elements of $X$, we say that $x \in X$ is an
\emph{$\mathcal E$-accumulation point} of $(x_i) _{i \in I} $ if and only if there
is $E \in \mathcal E$   such that 
$x $ is an
$ E$-accumulation point of $(x_i) _{i \in I} $
(in the sense of Definition \ref{seq}). 

We say that $X$ satisfies the \emph{$\mathcal E$-accumulation property}
if and only if every $I$-indexed sequence of elements of $X$
has some $\mathcal E$-accumulation point.
\item
 If $(F_i) _{i \in I} $ 
is an $I$-indexed sequence  of subsets of  $X$, 
we say that a  point
$x \in X$ is 
an \emph{$ \mathcal E$-limit point}  
 of the sequence
$(F_i) _{i \in I} $ if and only if,
for some $E \in \mathcal E$,
$x$ is an $ E$-limit point  
 of 
$(F_i) _{i \in I}$ (cf. Definition \ref{seqO}).

If $\mathcal F$ is a family
of subsets of $X$, we say that $X$ satisfies the \emph{$\mathcal F$-$\mathcal E$-accumulation property}
if and only if every $I$-indexed sequence of elements of $\mathcal F$ 
has some $\mathcal E$-limit point.
\end{enumerate} 
 \end{definition}   

Case (1) in Definition \ref{sequent}
is the particular case of (2) when $\mathcal F$ is taken 
to be the set of all singletons of $X$.

When  $\mathcal E = \{ E \} $
has just one member, Definitions \ref{sequent}(1)(2)
reduce to  Definitions \ref{seq} and \ref{seqO}, respectively.  

\begin{remark} \labbel{cc2}    
Notice that if in Definition \ref{sequent}(1)
we take $I= \omega $ and we let  $\mathcal E $
be the set of all non principal ultrafilters over $ \omega$,
we get still another equivalent formulation of countable compactness.
This is the reformulation of a nowadays standard fact (see, e. g., \cite{GS}).
The equivalence follows also from Remark \ref{bastfiltr} below,
and the fact (Remark \ref{classical}) that countable compactness is equivalent to
$\CAP_ \omega $.  More generally, if $\lambda$ is regular, and 
in Definition \ref{sequent}(1)
we take $I= \lambda  $ and  $\mathcal E $
 the set of all uniform ultrafilters over $ \lambda $,
we get an equivalent formulation of 
$ [ \lambda  , \lambda ]$-compactness, equivalently,
of $\CAP_ \lambda  $.
\end{remark}

We now show how to get the definition of sequential compactness as a particular
case of Definition \ref{sequent}(1).

\begin{definitions} \labbel{sc}   
As usual, if $ W \subseteq \omega$ is infinite, we let
 $[W ] ^ \omega$ denote the set of all infinite subsets of $W$.
If  $ Z \in [ \omega  ] ^ \omega$,
we let 
$F_Z = \{ W \subseteq \omega \mid |Z \setminus W| \text{ is finite}\} $,
that is, $F_Z$ is the filter on $ \omega$ generated by the Frechet filter on $Z$. 

We now get \emph{sequential compactness} if in  
 Definition \ref{sequent}(1) we take
$ I=\omega$, and  $\mathcal E = \{ F_Z \mid Z \in [ \omega  ] ^ \omega \} $.

With the above choice of $I$ and $\mathcal E$, and taking $\mathcal F = \mathcal O$ in \ref{sequent}(2) 
(that is, considering sequences
$(O_i) _{i \in I}$ 
of nonempty open sets of $X$), we get a notion called \emph{sequential pseudocompactness} in 
\cite{AMPRT}, and \emph{sequential feeble compactness} 
in \cite{stseq}. Notice that in \cite{AMPRT}
the $O_i$'s are requested to be pairwise disjoint; 
however, it can be shown  \cite{sps} that we get equivalent definitions,
whether or not we suppose the 
$O_i$'s to be disjoint.
  \end{definitions}

\begin{remark} \labbel{bastfiltr}
Suppose that each element of $\mathcal E$ is closed under supersets, 
and let $\mathcal E' = \{ F \subseteq \mathcal P(I) \mid F \text{ is a filter on $I$ and } F \subseteq E, \text{ for some } E \in  \mathcal E \} $.
Then some
 point $x$ is an
$\mathcal E$-accumulation point of some sequence 
$\mathbf x = (x_i) _{i \in I} $ 
if and only if 
 $x$ is  an $\mathcal E'$-accumulation point of $\mathbf x$.
Indeed,  $\mathcal E'$-accumulation trivially
implies $\mathcal E$-accumulation.
On the other direction, 
if 
 $x$ is  an $\mathcal E$-accumulation point of $\mathbf x$,
then there is $E \in \mathcal E$
such that $I _{\mathbf x,U}=\{ i \in I \mid x_i \in U\} \in E$,
for every open neighborhood $U$ of $x$.
If $F$  is the filter generated by 
$G=\{ I _{\mathbf x,U} \mid U \text{ is an open neighborhood of } x\}$,
then $F \subseteq E$, since 
$G$ is closed under intersection, and 
$E$ is closed under supersets.
Thus $F \in \mathcal E'$, and $F$  witnesses
that $x$ is  an $\mathcal E'$-accumulation point of $\mathbf x$
(cf. also Remarks \ref{minim} and \ref{resemb}).

In particular, under the above assumptions on $\mathcal E$ and $\mathcal E'$,
 a topological space satisfies the $\mathcal E$-accumulation property
if and only if it satisfies the $\mathcal E'$-accumulation property.
Thus, in contrast with Remark \ref{resemb}, and as far as Definition \ref{sequent}
is concerned, it is no loss of generality to assume that all members of
$\mathcal E$ are filters. Of course, the above observation applies
only in case we are not concerned with the cardinality of $\mathcal E$,
since, in the above situation, the cardinality of $\mathcal E'$ is generally
strictly larger than 
the cardinality of $\mathcal E$. 

Notice that the above argument
carries over even when we consider
$\mathcal E'' = \{ F \subseteq \mathcal P(I) \mid F \text{ is a filter on $I$ and, }
 \text{ for some } E \in  \mathcal E, \  F \subseteq E
\text{ and $F$  is maximal among the filters contained in } E \} $
(because every filter $F \subseteq E$ can be extended to a maximal 
filter with this property, using Zorn's Lemma).
Sometimes this turns out to be useful.
 \end{remark}

We now introduce the
generalization of Definitions \ref{BAcpn}
and  \ref{ABO}
which furnishes the equivalent of
Definition \ref{sequent} in terms of properties of open covers.

\begin{definition} \labbel{ABOgen}
Suppose that $A$ is a set, $B, G \subseteq \mathcal P(A)$, and $X$ is a topological space.
  \begin{enumerate}    
\item   
We say that $X$  is
$ [ B , G ]$-\emph{compact}
if and only if one of the following equivalent conditions hold.
  \begin{enumerate}    
\item 
If  $( O _ a) _{ a \in A } $ are open sets of $X$,
and, for every $K \in G$, 
 $( O _ a) _{ a \in K   } $
 is a cover of $X$, 
then there is  $H \in B  $ such that  
 $( O _ a) _{ a \in H } $
 is a cover of $X$, 
\item   
If  $( C _ a) _{ a \in A   } $ is a sequence 
 of closed subsets of $X$, and,
for every $H \in B  $,   
$ \bigcap _{ a \in H}  C_ a \not= \emptyset $,
then  there is $K \in G$ 
such that 
$ \bigcap _{ a \in K}  C_ a \not= \emptyset $.
\end{enumerate} 
\item 
If $\mathcal F$ is a family of subsets of  $X$, we say that $X$  is
$\mathcal F$-$ [ B , G ]$-\emph{compact}
if and only if one of the following equivalent conditions hold.
  \begin{enumerate}    
\item   
If  $( O _ a) _{ a \in A } $ are open sets of $X$,
and, for every $K \in G$, 
 $( O _ a) _{ a \in K   } $
 is  a cover of $X$, 
then there is  $H \in B  $ such that,  
for every  $F\in \mathcal F$, there is $a \in H$ such that   
$  O_ a \cap F\not= \emptyset $.
\item   
If  $( C _ a) _{ a \in A   } $
 are closed sets of $X$, and,
for every $H \in B  $, 
there exists $F\in \mathcal F$ such that   
$ \bigcap _{ a \in H}  C_ a \supseteq F$,
then  there is $K \in G$ 
such that 
$ \bigcap _{ a \in K}  C_ a \not= \emptyset $.
\end{enumerate} 
\end{enumerate} 
\end{definition}

Case (1) in Definition \ref{ABOgen}
is the particular case of (2) when $\mathcal F$ is taken 
to be the set of all singletons of $X$.

 Definitions \ref{BAcpn} and  \ref{ABO} are the particular cases of the above definition
 when $G= \{ A \}$.

\begin{remark} \labbel{schee}    Some known notions are particular cases of 
$ [ B , G ]$-\brfrt compactness,
as introduced in Definition \ref{ABOgen}. 

Indeed, in the particular case when 
$G$ is a partition of $A$, say into $\kappa$ classes, the hypothesis in 
Condition (1)(a) of  Definition \ref{ABOgen}
amounts exactly to considering a family of $\kappa$ open covers
of $X$, each cover having the same cardinality 
as the corresponding class. 
In the rest of this remark we shall deal only with the particular case
when $A$ is countable and $G$ is a partition of 
$A$ into $ \omega$-many classes, each class having cardinality $ \omega$. 

If, under the above assumptions,
 we let $B$ consist of all subsets of $A$ such that $B$ has finite
intersection with
each element of $G$,
then Condition (1)(a) in Definition \ref{ABOgen}
asserts that, given a countable family of countable covers of  $X$, we can extract 
 a cover of $X$ by selecting
a finite number of elements from each one of the original covers.
This property turns out to be equivalent to what nowadays is called the
\emph{Menger property for countable covers}, and is denoted by 
 ${\rm S _{fin}}  (\mathcal O, \mathcal O)$ in \cite[Section 5]{Sch}
(here we are following the notations from \cite{Sch}, and $\mathcal O$ denotes
the collection of all open covers of $X$). [Notice that the terminology
in the published journal version of this preprint is not standard. There we used
``Menger
property''  for ``Menger property for countable covers''.
Similarly for the Rothberger properties.]

On the other hand, if $B$ consists of all subsets of $A$ such that $B$ intersects
each element of $G$  in exactly one element, we get the
\emph{Rothberger property for countable covers}, denoted
by ${\rm S _{1}} (\mathcal O, \mathcal O)$   
in \cite[Section 6]{Sch}.
 
The connections between Definition \ref{ABOgen} and the 
notions introduced in \cite{Sch} probably deserve further analysis.
Notice that here we put no restriction on covers,
while \cite{Sch} also deals with special classes of covers,
such as \emph{large covers}, \emph{$\omega$-covers}  and so on.
One probably gets  interesting notions  modifying 
Definitions \ref{BAcpn},  \ref{ABOgen} etc.,  by putting
restrictions on the nature of the starting cover and of the resulting subcover.
This suggests the next definition.
\end{remark}

\begin{definition} \labbel{ababg}
Suppose that $A$ is a set, $B, G \subseteq \mathcal P(A)$, $X$ is a topological space, and $\mathcal A $, $ \mathcal B$ are collections of covers of $X$.
 
$X$  is
$ [ B _{\mathcal B} , G_{\mathcal A} ]$-\emph{compact}
if and only if whenever  $( O _ a) _{ a \in A } $ are subsets of $X$,
and, for every $K \in G$, 
 $( O _ a) _{ a \in K   } $
 is a cover in $\mathcal A$, 
then there is  $H \in B  $ such that  
 $( O _ a) _{ a \in H } $
 is a cover in $\mathcal B$.
 \end{definition}   

Arguing as in Remark \ref{schee}, the properties
$ {\rm S _{fin}} (\mathcal A, \mathcal B)$
and 
${\rm S _{1}}(\mathcal A, \mathcal B)$
from \cite{Sch} are particular cases of Definition \ref{ababg}. 

The particular case of Definition \ref{ABOgen} in which
$A=\lambda$, $G$ is the set of subsets of $\lambda$ of cardinality $\lambda$,
and $B= \mathcal P _ \kappa ( \lambda )$  
has been briefly hinted on \cite[p. 1380]{tapp}
under the name \emph{almost $ [ \kappa , \lambda  ]$-compactness}. 

In the next theorems 
we give the connections between the notions
introduced in Definitions \ref{sequent} and \ref{ABOgen}.

Recall the definition of $a^<_B$ given just
before Theorem \ref{e}.

\begin{theorem} \labbel{eG} 
Suppose that $A$ is a set, $B, G \subseteq \mathcal P(A)$, and $X$ is a topological space. Then the following conditions are equivalent.
\begin{enumerate} 
\item
$X$ is $ [ B, G]$-\brfrt compact.

\item
For every  sequence  $( P _ a ) _{ a \in A} $
of subsets of $X$, if,
for every $H \in B$, 
$ \bigcap _{ a \in H}  P_ a \not= \emptyset $,
 then  there is $K\in G$ such that  
$ \bigcap _{ a \in K}  \overline{P}_ a \not= \emptyset $.

\item
For every  sequence 
$\{ x_H \mid H \in B \}$
of elements of  $X$,
  there is $K\in G$ such that  
    $ \bigcap _{ a \in K} 
 \overline{ \{ x_H \mid H \in a^<_B \}}
  \not= \emptyset $.
 
\item
For every  sequence 
$\{ Y_H \mid H \in B \}$
of nonempty subsets of  $X$,
  there is $K\in G$ such that  
    $ \bigcap _{ a \in K} 
 \overline{ \bigcup \{ Y_H \mid H \in a^<_B \}}
  \not= \emptyset $.

\item
$X$ satisfies the $\mathcal E$-accumulation property, for
$I=B$ and
$\mathcal E = \{ E_K \mid K \in G \} $ where, for 
$K\in G$, we put  
 $E_K= \{ Z \subseteq B \mid \text{for every } a \in K
\text{ there is }
 H \in Z \text{ such that } a \in H \} $.
 \end{enumerate}
 \end{theorem}

\begin{theorem} \labbel{eOG} 
Suppose that $A$ is a set, $B, G \subseteq \mathcal P(A)$,  $X$ is a topological space, and $\mathcal F$ is a family of subsets of $X$. Then the following conditions are equivalent.
\begin{enumerate} 
\item
$X$ is $\mathcal F$-$ [ B, G]$-\brfrt compact.
\item
For every sequence  $( P _ a) _{ a \in A   } $
 of subsets of $X$, if,
for every $H \in B  $, 
there exists $F\in \mathcal F$ such that   
$ \bigcap _{ a \in H}  P_ a \supseteq F$,
then  
  there is $K\in G$ such that  
$ \bigcap _{ a \in K}  \overline{P}_ a \not= \emptyset $.
\item
For every  sequence 
$\{ F_H \mid H \in B \}$
of elements of $\mathcal F$,
  there is $K\in G$ such that   
   $ \bigcap _{ a \in K} 
 \overline{ \bigcup \{ F_H \mid H \in a^<_B \}}
  \not= \emptyset $.
\item
For every  sequence 
$\{ Y_H \mid H \in B \}$
of subsets of  $X$
such that each $Y_H$ contains some 
$F_H \in \mathcal F$, 
  there is $K\in G$ such that  
    $ \bigcap _{ a \in K} 
 \overline{ \bigcup \{ Y_H \mid H \in a^<_B \}}
  \not= \emptyset $.
\item 
 $X$ satisfies the $\mathcal F$-$\mathcal E$-accumulation property, for
$I$ and $\mathcal E$ as in Condition \ref{eO}(5) above.
\end{enumerate}
 \end{theorem}

When $G= \{ A \}$, the conditions in Theorems \ref{eG} and \ref{eOG}  
turn out to coincide with the corresponding conditions in Theorems \ref{e} and \ref{eO}.

\begin{theorem} \labbel{rEE}
Suppose that $X$ is a topological space, $I$ is a set,
$G$ is a set of subsets of
$ \mathcal P (I)$,
and put $\mathcal E= \{ K^+ \mid K \in  G \} $
and $ A= \bigcup \mathcal E $.
Then the following conditions are equivalent.
  \begin{enumerate}   
 \item   
$X$ satisfies the $\mathcal E$-accumulation property.
\item
If $(O_a) _{a \in A} $ are open sets of $X$, and,
 for every $K \in G$,
 $(O_a) _{a \in K} $ is  a cover of $X$,
then there is $i \in I$
such that  
 $(O_a) _{i \in a \in A} $ is a cover of $X$.
\item 
$X$ is $ [ B , G]$-\brfrt compact, where $B= \{i^<_A \mid i \in I\} $.
  \item  
For every sequence
$(x_i) _{i \in I} $  of elements of $X$, 
there is $K \in G$ such that  
if, for each $a \in K$,  we put 
$C_a = \overline{\{ x_i \mid i \in a \}}   $, 
then 
$\bigcap _{a \in K}C_a \not= \emptyset  $.  
\end{enumerate} 
 \end{theorem} 

\begin{proof}
Similar to
the proof of Theorem \ref{r}.
Notice that
(2) $\Leftrightarrow $  (3) is immediate from the definitions, using
Condition (1)(a) in Definition \ref{ABOgen}, and that
(1) $\Leftrightarrow $  (4) follows directly from Proposition \ref{ae}. 
\end{proof}

\begin{theorem} \labbel{rEEf}
Under the assumptions in Theorem \ref{rEE},
and if $\mathcal F$ is a family of subsets of $X$, 
then the following conditions are equivalent.
  \begin{enumerate}   
 \item   
$X$ satisfies the $\mathcal F$-$\mathcal E$-accumulation property.
\item 
$X$ is $\mathcal F$-$ [ B , G]$-\brfrt compact, where $B= \{i^<_A \mid i \in I\} $.
  \item  
For every sequence
$(F_i) _{i \in I} $  of elements of $\mathcal F$, 
there is $K \in G$ such that  
if, for each $a \in K$,  we put 
$C_a = \overline{ \bigcup_{i \in a} F_i }   $, 
then 
$\bigcap _{a \in K}C_a \not= \emptyset  $.  
\end{enumerate} 
 \end{theorem} 

Theorem \ref{rEE} is the particular case of Theorem \ref{rEEf}
when $\mathcal F$ is the family of all singletons of $X$. 
Theorems \ref{r} and \ref{rO}  
are the particular cases of, respectively,  Theorems \ref{rEE} and \ref{rEEf} 
when $G= \{ A\} $  
has just one member.

The following characterization of sequential compactness 
in terms of open covers might be known, but
we know no reference for it.

\begin{corollary} \labbel{seqopcov}
A topological space $X$ is sequentially
compact (sequentially feebly compact, respectively)
if and only if, for every  open cover 
$\{ O_a \mid a \in [ \omega ] ^ \omega \}$ of $X$ 
such that 
$\{ O_a \mid a \in [ Z ] ^ \omega \}$ 
is still a cover of $X$, for every $Z \in [ \omega ] ^ \omega$,
there is $n \in  \omega$  
such that $\{ O_a \mid n \in a \in [ \omega ] ^ \omega \}$
is a cover of $X$  (is dense in $X$, respectively).
 \end{corollary}

 \begin{proof}
Take $I= \omega $
and $G = \{[Z] ^ \omega  \mid Z \in  [ \omega ] ^ \omega  \}$ 
in Theorems \ref{rEE} and \ref{rEEf}.
If $K=[Z] ^ \omega\in G$, then 
$K^+= F_Z$, in the notations of  
Definition \ref{sc}. 
Thus the corollary is a particular case of
the equivalence (1) $\Leftrightarrow $  (2) in
Theorems \ref{rEE} and \ref{rEEf}, respectively,

Of course, also a direct proof of Corollary \ref{seqopcov} is not difficult. 
 \end{proof}

As a special case of Theorem \ref{eG} (1) $\Leftrightarrow $  (3), we get 
the following characterizations (probably folklore) 
of the Rothberger and the Menger properties for countable covers.

\begin{corollary} \labbel{roth}
A topological space $X$ satisfies the 
Rothberger property for countable covers if and only if, 
for every sequence 
$ \{ x _f \mid f: \omega \to \omega \}  $
of elements of $X$,  
there is $ n \in \omega$ 
such that 
$\bigcap _{m \in \omega }
   \overline{ \{ x_f \mid f(n)= m\} }  \not= \emptyset  $.  

A topological space $X$ satisfies the 
Menger property for countable covers if and only if, 
for every sequence 
$ \{ x _f \mid f: \omega \to [\omega] ^{< \omega }  \}  $
of elements of $X$,  
there is $ n \in \omega$ 
such that 
$\bigcap _{m \in \omega }
   \overline{ \{ x_f \mid m \in f(n)\} }  \not= \emptyset  $.  
 \end{corollary}

If $I$ is a set, and $M$ is a set of ultrafilters over  $I$, then
a topological space $X$ is said to be \emph{quasi $M$-compact} if and only if,
for every $I$-indexed sequence $(x_i) _{i \in I}$ 
of elements of $X$, there exists $D \in M$ such that 
$(x_i) _{i \in I}$ $D$-converges to some point of $X$. 
Of course, if $M=\{ D \}$ is a singleton, then 
quasi $M$-compactness is the same as $D$-compactness,
and is also equivalent to weak  $M$-compactness (Definition \ref{weakly}).
See \cite{GFwc} for further references about these notions.

 The space 
$X$ is \emph{quasi $M$-pseudocompact} if and only if,
for every $I$-indexed sequence $(O_i) _{i \in I}$ 
of nonempty open sets of $X$, there exists $D \in M$ such that 
$(O_i) _{i \in I}$ has some $D$-limit point in $X$.
Notice that, for $I= \omega $, the above notion is called  
\emph{$M$-pseudocompactness} in \cite[Definition 2.1]{GFsomegen}.
We have chosen the name  
quasi $M$-pseudocompactness in analogy with
quasi $M$-compactness.

\begin{corollary} \labbel{quasicp}
Suppose that  $M$ is a set of ultrafilters over some set $I$, 
and let $A = \bigcup_{D\in M} D$. 
Then
a topological space $X$ is quasi $M$-compact (quasi $M$-pseudocompact, respectively)
if and only if,
whenever $(O_a) _{a \in A} $ are open sets
of $X$, 
and, for every $D \in M$,
 $(O_a) _{a \in D} $ is a cover of $X$,
then there is $i \in I$
such that  
 $(O_a) _{i \in a \in A} $ is a cover of $X$
(is dense in $X$, respectively).
 \end{corollary} 

\begin{proof} 
By 
Theorems \ref{rEE} and \ref{rEEf}  (1) $\Leftrightarrow $  (2),
with $\mathcal E = M$, since, as already noticed, if $D$ is an ultrafilter,
then $D^+=D$. 
\end{proof}

\begin{remark} \labbel{hon}
As a final remark, let us mention that not every ``covering property'' present in the literature has
the form given in Definitions \ref{BAcpn}, \ref{ABO}, or \ref{ABOgen},
the most notable case being paracompactness. 
More generally, all covering properties
involving some particular properties (local finiteness, point finiteness, etc.) of the 
original cover, or of the resulting subcover,
 are not part of  the framework given by Definition \ref{BAcpn},
as it stands.

There are even   equivalent formulations of countable compactness
 which, at least formally, are not  particular cases  of
Definition \ref{BAcpn}.
Indeed, a space $X$ is countably compact
if and only if, for  every countable open cover 
$(O_n) _{n \in \omega } $
such that $O_n \subseteq O_m$, for $n \leq m < \omega $,
there is $ n \in \omega$  such that 
$O_n= X$.
The above condition cannot be directly expressed as a particular case of 
Definition \ref{BAcpn}.
 
In spite of the above remarks, we believe to have demonstrated
that Definition \ref{BAcpn} and its variants are  general enough to capture many 
disparate and seemingly unrelated
notions, being  at the same time sufficiently
concrete and manageable so that interesting results
can be proved about them.

Of course, as we did in Definition \ref{ababg}, there is the possibility of modifying Definition \ref{BAcpn} 
and its variants by considering particular covers with special properties (cf. also Remark \ref{schee}).
We have not yet pursued this promising line of research.
 \end{remark}

\end{document}